\documentclass[12pt]{article}
 \usepackage{amsmath,amssymb,amscd,amsthm,esint}
 
\usepackage{graphics,amsmath,amssymb,amsthm,mathrsfs}

\usepackage{graphics,amsmath,amssymb,amsthm,mathrsfs,amsfonts}

\usepackage{sidecap}
\usepackage{float}
\usepackage{extarrows}
\usepackage{booktabs}
\usepackage{verbatim}
\usepackage{hyperref}
\usepackage[usenames,dvipsnames]{xcolor}

\newtheorem{thm}{Theorem}[section]
\newtheorem{lemma}[thm]{Lemma}

\theoremstyle{definition}
\newtheorem{remark}[thm]{Remark}

\def\Xint#1{\mathchoice
      {\XXint\displaystyle\textstyle{#1}}%
      {\XXint\textstyle\scriptstyle{#1}}%
      {\XXint\scriptstyle\scriptscriptstyle{#1}}%
      {\XXint\scriptscriptstyle\scriptscriptstyle{#1}}%
                   \!\int}
\def\XXint#1#2#3{{\setbox0=\hbox{$#1{#2#3}{\int}$}
         \vcenter{\hbox{$#2#3$}}\kern-.5\wd0}}

\def\dashint{\Xint-}
\def\R{\mathbb{R}}
\def\e{\varepsilon}

\numberwithin{equation}{section}

\begin{document}

\title{Convergence Rates in Parabolic Homogenization \\
with Time-Dependent Periodic Coefficients}

\author{ Jun Geng\thanks{Supported in part by the NNSF of China (11571152) and Fundamental Research Funds for the Central
Universities (LZUJBKY-2015-72).}\qquad Zhongwei Shen\thanks{Supported in part by NSF grant DMS-1161154.}}
\date{}
\maketitle

\begin{abstract}

For a family of second-order parabolic systems with bounded measurable,
 rapidly oscillating and time-dependent periodic coefficients,
 we investigate the sharp  convergence rates  of weak solutions in $L^2$.
 Both initial-Dirichlet and initial-Neumann  problems are studied.

\end{abstract}

\section{\bf Introduction}
The primary purpose of this paper is to investigate the sharp convergence rates in $L^2$ 
for a family of second-order parabolic operators $\partial_t+\mathcal{L}_\varepsilon$ with bounded measurable, 
rapidly oscillating and time-dependent periodic coefficients.
Both the initial-Dirichlet and initial-Neumann boundary value problems are studied.
Specifically, we consider
\begin{equation}\label{elliptic operator}
\mathcal{L}_\varepsilon=-\text{div}\left(A\big({x}/{\varepsilon},{t}/{\varepsilon^2}\big)\nabla
\right),
\end{equation}
where $\e>0$ and $A(y,s)= \big(a_{ij}^{\alpha\beta} (y,s)\big)$
with $1\le i, j\le d$ and $1\le \alpha, \beta \le m$.
Throughout this paper we will assume that the coefficient matrix
$A(y,s)$ is real, bounded measurable, and satisfies the ellipticity condition,
\begin{equation}\label{ellipticity}
\mu |\xi|^2\le  a^{\alpha\beta}_{ij}(y,s)\xi_i^\alpha
\xi^\beta_j\leqslant \frac{1}{\mu}|\xi|^2
\quad \text{ for any }
\xi=(\xi_i^\alpha ) \in \mathbb{R}^{m\times d} \text{ and  a.e. }  (y,s)\in \mathbb{R}^{d+1},
\end{equation}
where $\mu >0$,
and the periodicity condition,
\begin{equation}\label{periodicity}
A(y+z,s+t)=A(y,s)~~~\text{ for }(z,t)\in \mathbb{Z}^{d+1}\text{ and a.e. }(y,s)\in \mathbb{R}^{d+1}.
\end{equation}
No additional smoothness condition will be imposed on $A$.

Let $\Omega\subset\mathbb{R}^d$ be a bounded   domain and 
$0<T<\infty$.
We are interested in the initial-Dirichlet problem,
\begin{equation}\label{IDP}
\left\{\aligned
(\partial_t +\mathcal{L}_\e) u_\e  &=  F &\quad &\text { in } \Omega \times (0, T),\\
u_\e & = g & \quad & \text{ on } \partial\Omega \times (0, T),\\
u_\e  &=h &\quad &\text{ on } \Omega \times \{ t=0\},
\endaligned
\right.
\end{equation}
and the initial-Neumann problem, 
\begin{equation}\label{INP}
\left\{\aligned
(\partial_t +\mathcal{L}_\e) u_\e  &=  F &\quad &\text { in } \Omega \times (0, T),\\
\frac{\partial u_\e}{\partial \nu_\e} & = g & \quad & \text{ on } \partial\Omega \times (0, T),\\
u_\e  &=h &\quad &\text{ on } \Omega \times \{ t=0\},
\endaligned
\right.
\end{equation}
where $\displaystyle \left(\frac{\partial u_\e}{\partial \nu_\e}\right)^\alpha
=n_i a_{ij}^{\alpha\beta} (x/\e, t/\e^2)\frac{\partial u_\e^\beta}{\partial x_j}$ denotes the conormal derivative 
of $u_\e$ associated with $\mathcal{L}_\e$ and $n=(n_1, \dots, n_d)$ is the outward  normal to $\partial\Omega$.
Under suitable conditions on $F$, $g$, $h$ and $\Omega$, it is known that the weak solution $u_\e$ of
(\ref{IDP}) converges weakly in $L^2(0, T; H^1(\Omega))$ and
strongly in $L^2(\Omega_T)$ to $u_0$, where $\Omega_T =\Omega\times (0, T)$.
 Furthermore, the function $u_0$ is the weak solution of
the (homogenized) initial-Dirichlet problem,
\begin{equation}\label{IDP-0}
\left\{\aligned
(\partial_t +\mathcal{L}_0) u_0 &=  F &\quad &\text { in } \Omega \times (0, T),\\
u_0 & = g & \quad & \text{ on } \partial\Omega \times (0, T),\\
u_0  &=h &\quad &\text{ on } \Omega \times \{ t=0\}.
\endaligned
\right.
\end{equation}
Similarly, the weak solution $u_\e$ of
(\ref{INP}) converges weakly in $L^2(0, T; H^1(\Omega))$ and
strongly in $L^2(\Omega_T)$ to the weak solution of the (homogenized) initial-Neumann problem,
\begin{equation}\label{INP-0}
\left\{\aligned
(\partial_t +\mathcal{L}_0) u_0  &=  F &\quad &\text { in } \Omega \times (0, T),\\
\frac{\partial u_0}{\partial \nu_0} & = g & \quad & \text{ on } \partial\Omega \times (0, T),\\
u_0  &=h &\quad &\text{ on } \Omega \times \{ t=0\}.
\endaligned
\right.
\end{equation}
The operator $\mathcal{L}_0$ in (\ref{IDP-0}) and (\ref{INP-0}),
called the homogenized operator, 
is a second-order elliptic operator with constant coefficients \cite{bensoussan-1978}.

The following are the main results of the paper,
which establish  the sharp  $O(\e)$ convergence rates in $L^2(\Omega_T)$ for
both the initial-Dirichlet and the initial-Neumann problems.

\begin{thm}\label{main-theorem-1}
Suppose that the coefficient matrix $A$ satisfies (\ref{ellipticity}) and (\ref{periodicity}).
Let $\Omega$ be a bounded $C^{1,1}$ domain in $\R^d$.
 Let $u_\e, u_0\in L^2(0, T; H^1(\Omega))$ be   weak solutions of
 (\ref{IDP}) and (\ref{IDP-0}), respectively,   for some $F\in L^2(\Omega_T)$.
 Assume that $u_0\in L^2(0, T; H^2(\Omega))$.
 Then 
 \begin{equation}\label{main-estimate-1}
 \aligned
& \|   u_\e -u_0\|_{L^2(\Omega_T)}\\
& \le C \e \left\{ \|  u_0\|_{L^2(0, T; H^2(\Omega))}
 +\| F\|_{L^2(\Omega_T)}
 +\sup_{\e^2 <t<T} \left(\frac{1}{\e} \int_{t-\e^2}^t \int_\Omega |\nabla u_0|^2\right)^{1/2} \right\},
 \endaligned
 \end{equation}
 where $C$ depends at most on $d$, $m$, $\mu$,  $T$ and $\Omega$.
 \end{thm}
 
 \begin{thm}\label{main-theorem-2}
 Let $u_\e\in L^2(0, T; H^1(\Omega))$ be  a weak solution of
 (\ref{INP})  for some $F\in L^2(\Omega_T)$
 and $u_0\in L^2(0, T; H^1(\Omega))$ the weak solution of the  homogenized problem (\ref{INP-0}).
 Under the same assumptions as in Theorem \ref{main-theorem-1},
 the estimate (\ref{main-estimate-1}) holds.
 \end{thm}
 
 \begin{remark}\label{remark-1.1}
 {\rm 
In Theorems \ref{main-theorem-1} and \ref{main-theorem-2}
 we do not specify the conditions directly on $g$ and $h$, but rather require 
 $u_0 \in L^2(0, T; H^2(\Omega))$.
 In the case that either
 $u_\e=u_0=0$ or $\displaystyle \frac{\partial u_\e}{\partial \nu_\e}=\frac{\partial u_0}{\partial \nu_0}=0$
 on $\partial\Omega \times (0, T)$, i.e. $g=0$,
 the third term in the r.h.s. of (\ref{main-estimate-1}) may be bounded
 by 
 $$
 C\big\{ \|\partial_t u_0\|_{L^2(\Omega_T)} + \| F\|_{L^2(\Omega_T)} +\| h\|_{L^2(\Omega)}\big\}.
 $$
See (\ref{3.10-2}).
 As a result, we obtain 
 \begin{equation}\label{main-estimate-2}
 \| u_\e -u_0\|_{L^2(\Omega_T)}
 \le C\, \e \Big\{ \| u_0\|_{L^2(0, T; H^2(\Omega))} +\| F\|_{L^2(\Omega_T)} +\| h\|_{L^2(\Omega)} \Big\},
 \end{equation}
 where $C$ depends at most on $d$, $m$, $\mu$, $T$ and $\Omega$. In particular, if
 $g=0$ and $h=0$, then 
 $$
 \| u_0\|_{L^2(0, T; H^2(\Omega))} \le C \| F\|_{L^2(\Omega_T)}
 $$
  (see (\ref{3.10-5})). It follows that 
\begin{equation}\label{main-estimate-3}
\| u_\e -u_0\|_{L^2(\Omega_T)}
 \le C\, \e \| F\|_{L^2(\Omega_T)}.
 \end{equation}
 Also, in the case that $g=0$ on $\partial\Omega\times (0, T)$ and $h\in H^1(\Omega)$,
 it is known that if $\mathcal{L}_0^* =\mathcal{L}_0$, then
 $$
 \| u_0\|_{L^2(0, T; H^2(\Omega))} 
 \le C \Big\{ \| F\|_{L^2(\Omega_T)} +\| h\|_{H^1(\Omega)} \Big\}
 $$
 \cite{Lady}. This gives
  \begin{equation}\label{main-estimate-4}
 \| u_\e -u_0\|_{L^2(\Omega_T)}
 \le C\, \e \Big\{ \| F\|_{L^2(\Omega_T)} +\| h\|_{H^1(\Omega)} \Big\},
 \end{equation}
 where $C$ depends at most on $d$, $m$, $\mu$,  $T$ and $\Omega$.
 }
 \end{remark}
 
 The sharp convergence rate is one of the central issues  in quantitative homogenization 
 and has been studied extensively in the various  settings.
 For elliptic equations and systems
 in divergence form with periodic coefficients,
 related results may be found in the recent work  \cite{
 Suslina-2012, Suslina-2013, KLS2, KLS3, KLS4, Shen-Boundary-2015, Gu-2015, SZ-2015}
 (also see  \cite{bensoussan-1978, Jikov-1994, Griso-2004, Griso-2006, Onofrei-2007}
 for references on earlier work).
 In particular, the order sharp estimate 
 \begin{equation}\label{elliptic-estimate}
 \| u_\e -u_0\|_{L^2(\Omega)} \le C\,\e \| F\|_{L^2(\Omega)},
 \end{equation}
 holds, if $\mathcal{L}_\e (u_\e)=\mathcal{L}_0 (u_0)=F$ in $\Omega$ and
 $u_\e =u_0=0$ or $\frac{\partial u_\e}{\partial\nu_\e}=\frac{\partial u_0}{\partial \nu_0}=0$ on $\partial\Omega$
 (see \cite{Suslina-2012, Suslina-2013, Gu-2015, SZ-2015} for $C^{1,1}$ domains
 and \cite{KLS2, KLS4, Shen-Boundary-2015} for Lipschitz domains).
 For parabolic equations and systems various  results are known in the case where 
 the coefficients are time-independent \cite{Jikov-1994, Suslina-2004, Zhikov-2006,
 Suslina-2015}. 
 We note that in this case, using the partial Fourier transform in the $t$ variable,
  it is possible to represent the solution of the parabolic system
  as an integral of the resolvent of the elliptic operator $\mathcal{L}_\e$
  and apply the elliptic estimates.
 
 Very few results are known if the coefficients are
  time-dependent.
In fact,  to the authors' best knowledge, 
 the only known estimate in this case is
 \begin{equation}\label{parabolic-max}
 \| u_\e -u_0\|_{L^\infty(\Omega_T)} \le C \e,
 \end{equation}
 obtained by the use of the maximum principle, 
 where $C$ depends on $u_0$ and coefficients are assumed to be smooth \cite{bensoussan-1978}.
 Our  order sharp estimates (\ref{main-estimate-2})-(\ref{main-estimate-4}), which extend (\ref{elliptic-estimate}) to the parabolic setting,
 seem to be the first work in this area beyond the rough estimate (\ref{parabolic-max}).
 
 We now describe some of key ideas in the proof of Theorems \ref{main-theorem-1} and \ref{main-theorem-2}.
 Although it is not clear how  to reduce parabolic systems with time-dependent coefficients to elliptic systems
 by some simple transformations, 
 our general approach to the estimate (\ref{main-estimate-1})
  is inspired by the work on elliptic systems
 mentioned above.
 We consider the function
 \begin{equation}\label{w-1}
 w_\e =u_\e (x, t)-u_0 (x, t)
 -\e \chi (x/\e, t/\e^2)  K_\e (\nabla u_0) 
 -\e^2  \phi (x/\e, t/\e^2) \nabla K_\e (\nabla u_0),
 \end{equation}
 where $\chi (y, s)$ and $\phi(y,s)$ are correctors and dual correctors for
 the family of operators $\partial_t +\mathcal{L}_\e$, $\e>0$ (see Section 2 for their definitions).
 In (\ref{w-1}) the operator $K_\e: L^2(\Omega_T) \to C_0^\infty(\Omega_T)$ is a parabolic smoothing
 operator at scale $\e$.
 We note that in the elliptic case \cite{Suslina-2012,Suslina-2013,Shen-Boundary-2015,SZ-2015},
 only the first three terms in the r.h.s. of (\ref{w-1}) are used.
 By computing $ (\partial_t +\mathcal{L}_\e) w_\e$, we are able to show that
 \begin{equation}\label{1.1-dual}
 \aligned
  &\Big|\int_0^T \langle \partial_t w_\e, \psi\rangle +\iint_{\Omega_T}
 A^\e\nabla w_\e \cdot \nabla \psi \Big|\\
 & \le C \Big\{ \| u_0\|_{L^2(0, T; H^2(\Omega))}
 +\|\partial_t u_0\|_{L^2(\Omega_T)}
 +\e^{-1/2}\| \nabla u_0\|_{L^2(\Omega_{T, \e})}  \Big\}\\
& \qquad \qquad \qquad \qquad 
\cdot \Big\{ \e \| \nabla \psi\|_{L^2(\Omega_T)} +\e^{1/2}
 \|\nabla \psi\|_{L^2(\Omega_{T, \e})} \Big\}
 \endaligned
 \end{equation}
 for any $\psi \in L^2(0, T; H^1_0(\Omega))$ in the case of Dirichlet  condition (\ref{IDP}),
 and for any $\psi \in L^2(0, T; H^1(\Omega))$ in the case of the Neumann  condition (\ref{INP}),
 where $\Omega_{T, \e}$ denotes the set of points in  $\Omega_T$ whose (parabolic) distances
 to the boundary of $\Omega_T$ are less than $\e$ 
 (see Section 3 for details).
 By taking $\psi =w_\e$ in (\ref{1.1-dual}) we obtain an $O(\sqrt{\e})$ error estimate 
 in $L^2(0, T; H^1(\Omega))$, 
 \begin{equation}\label{H-1}
 \| \nabla w_\e\|_{L^2(\Omega_T)}
 \le C \sqrt{\e}
  \Big\{ \| u_0\|_{L^2(0, T; H^2(\Omega))}
 +\|\partial_t u_0\|_{L^2(\Omega_T)}
 +\e^{-1/2} \| \nabla u_0\|_{L^2(\Omega_{T, \e})}  \Big\},
 \end{equation}
 which is more or less sharp,
  for both the initial-Dirichlet and the initial-Neumann problems.
  Finally, with (\ref{1.1-dual}) at our disposal,
  we give the proof of Theorems \ref{main-theorem-1} and \ref{main-theorem-2} in Section 4.
  This is done by a dual argument, inspired by \cite{Suslina-2012, Suslina-2013}.
 
We point out that results on convergence rates 
are useful in the study of regularity estimates that are uniform in $\e>0$
\cite{Armstrong-Smart-2014, Armstrong-Shen-2016, Shen-Boundary-2015}.
For solutions of $(\partial_t +\mathcal{L}_\e)u_\e=F$, the uniform boundary H\"older
and interior Lipschitz estimates were proved in  \cite{Geng-Shen-2015} by a
compactness method, introduced to the study of homogenization problems in \cite{AL-1987}.
The results obtained in this paper should allow us to establish the boundary Lipschitz estimates
as well as Rellich estimates at large scale for parabolic systems
 in a manner similar to that in \cite{Shen-Boundary-2015} for elliptic systems of linear elasticity. 
We plan to carry this out in a separate study.

We end this section with some notations that will be used throughout the paper.
A function $h=h(y,s)$ in $\R^{d+1}$
 is said to be $1$-periodic if $h$ is periodic with respect to $\mathbb{Z}^{d+1}$. 
 We will use the notation
 $$
 h^\e (x,t)= h (x/\e, t/\e^2)
 $$
 for $\e>0$, and the summation convention that the repeated indices are summed.
 Finally, we use $C$ to denote constants that depend at most on $d$, $m$, $\mu$,
 $T$ and $\Omega$, but never on $\e$.
 
%%%%%%%%%%%%%%%%%%%%%%%%%%%%%%%%%%%%%%%%%%%%%%%%%%

%%%%%%%%%%%%%%%%%%%%%%%%%%%%%%

\section{\bf Correctors and dual correctors}

 Let $\mathcal{L}_\varepsilon=-\text{div}\left(A^\e(x,t) \nabla\right)$, where $A^\e (x,t)=A(x/\e, t/\e^2)$
 and $A(y, s)$ is 1-periodic and satisfies   the ellipticity condition (\ref{ellipticity}).
 For $1\leq j\leq d$ and $1\le \beta\le m$, 
 the corrector $\chi_j^\beta=\chi_j^\beta (y,s)=(\chi_{j}^{\alpha\beta} (y, s))$ 
  is defined as the weak solution of the following cell problem:
\begin{equation}\label{corrector}
\begin{cases}
\big(\partial_s +\mathcal{L}_1\big) (\chi_j^\beta)
=-\mathcal{L}_1(P_j^\beta ) ~~~\text{in}~~Y, \\
\chi_j^\beta =\chi^\beta_j(y,s)~~ \text{is } \text{1-periodic in } (y,s),\\
\int_{Y} \chi_j^\beta = 0,
\end{cases}
\end{equation}
where $Y=[0,1)^{d+1}$, $P_j^\beta  (y)=y_j e^\beta$, and
$e^\beta=(0, \dots, 1, \dots, 0 )$ with $1$ in the $\beta^{th}$ position. 
Note that 
\begin{equation}
(\partial_s+\mathcal{L}_1)(\chi_j^\beta +P_j^\beta )=0~~~\text{in}~~\mathbb{R}^{d+1}.
\end{equation}
By the rescaling property of $\partial_t +\mathcal{L}_\e$, 
one obtains that
\begin{equation}
(\partial_t+\mathcal{L}_\varepsilon)\left\{\varepsilon\chi_j^\beta(x/\varepsilon,t/\varepsilon^2)+P^\beta_j(x)\right\}
=0~~~\text{in}~~\mathbb{R}^{d+1}.
\end{equation}

Let $\widehat{A}=(\widehat{a}^{\alpha\beta}_{ij})$, where $1\leq i,j\leq d$, $1\le \alpha, \beta\le m$,
and
\begin{align}\label{A}
\widehat{a}^{\alpha\beta}_{ij}=\dashint_{Y}\left[a^{\alpha\beta}_{ij}+a^{\alpha\gamma}
_{i k}\frac{\partial}{\partial y_k}\chi^{\gamma\beta}_{j}\right];
\end{align}
that is
$$
\widehat{A}=\dashint_Y \Big\{ A +A\nabla \chi \Big\}.
$$
It is known that the constant matrix  $\widehat{A}$ satisfies the ellipticity condition,
$$
\mu |\xi|^2 \le \widehat{a}_{ij}^{\alpha\beta} \xi_i^\alpha \xi_j^\beta
\le \mu_1 |\xi|^2 \qquad \text{ for any } \xi=(\xi_j^\beta) \in \R^{m\times d},
$$
where $\mu_1>0$ depends only on $d$, $m$ and $\mu$ \cite{bensoussan-1978}.
Denote $\mathcal{L}_0=-\text{div}(\widehat{A}\nabla)$. Then $\partial_t+\mathcal{L}_0$ is the homogenized operator 
for the family of parabolic operators  $\partial_t+\mathcal{L}_\varepsilon$, $\e>0$.

To introduce the dual correctors, we consider the 1-periodic matrix-valued function
\begin{equation}\label{B}
B= A + A\nabla \chi -\widehat{A}.
\end{equation}
More precisely,  $B=B(y, s)= \big( b_{ij}^{\alpha\beta}\big)$, where
$1\le i, j\le d$, $1\le \alpha, \beta\le m$, and
\begin{equation}\label{b}
b_{ij}^{\alpha\beta}=a_{ij}^{\alpha\beta}+a_{ik}^{\alpha\gamma}
\frac{\partial \chi^{\gamma\beta}_j}{\partial y_k}-\widehat{a}^{\alpha\beta}_{ij}.
\end{equation}

\begin{lemma}\label{1.15}
Let $1\leq j\leq d$ and $1\le \alpha, \beta\le m$.
Then there exist 1-periodic functions 
$\phi_{kij}^{\alpha\beta}(y,s)$ in $\R^{d+1}$ such that 
$\phi_{kij}^{\alpha\beta}\in H^1(Y)$,
\begin{equation}\label{1.10}
b_{ij}^{\alpha\beta}=\frac{\partial}{\partial y_k}(\phi^{\alpha\beta}_{kij})~~\text{ and }~~\phi^{\alpha\beta}_{kij}
=-\phi_{ikj}^{\alpha\beta},
\end{equation}
where $1\le k, i\le d+1$, $b_{ij}^{\alpha\beta}$ is defined by
(\ref{b}) for $1\le i\le d$,  $b_{(d+1)j}^{\alpha\beta}=-\chi_j^{\alpha\beta}$,
and we have used the notation $y_{d+1}=s$.
\end{lemma}

\begin{proof}
Observe  that by (\ref{corrector}) and (\ref{A}), $b_{ij}^{\alpha\beta} \in L^2 (Y)$ and 
\begin{equation}\label{1.8-1}
\int_{Y} b_{ij}^{\alpha \beta}=0
\end{equation}
for $1\le i\le d+1$.
It follows that  there exist $f_{ij}^{\alpha\beta} \in H^2 (Y)$ such that
\begin{equation}\label{1.8-2}
\begin{cases}
\Delta_{d+1} f_{ij}^{\alpha\beta}=b^{\alpha\beta}_{ij}~~~~\text{ in } \R^{d+1}, \\
f_{ij}^{\alpha\beta} ~~\text{is  1-periodic}~~~\text{ in } \R^{d+1}, 
\end{cases}
\end{equation}
where $\Delta_{d+1}$ denotes  the Laplacian in $\R^{d+1}$. Write 
\begin{align}\label{1.11}
b_{ij}^{\alpha\beta}=
\frac{\partial}{\partial y_k}\left\{\frac{\partial}{\partial y_k}f_{ij}^{\alpha\beta}
-\frac{\partial}{\partial y_i}f_{kj}^{\alpha\beta}\right\}+\frac{\partial}{\partial y_i}\left\{\frac{\partial}{\partial y_k}f_{kj}^{\alpha\beta}\right\},
\end{align}
where the index $k$ is summed from $1$ to $d+1$.
Note that by (\ref{corrector}),
 \begin{equation}\label{div}
 \sum_{i=1}^{d+1} \frac{\partial b^{\alpha\beta}_{ij}}{\partial y_i}
 =\sum_{i=1}^d \frac{\partial}{\partial y_i}b^{\alpha\beta}_{ij}-\frac{\partial}{\partial s}\chi_j^{\alpha\beta}=0.
 \end{equation}
In view of (\ref{1.8-2}) this implies that 
$$
\sum_{i=1}^{d+1}\frac{\partial }{\partial y_i}  f_{ij}^{\alpha\beta}
$$
is harmonic in $\R^{d+1}$. Since it is 1-periodic, it must be constant.
Consequently, by (\ref{1.11}), we obtain 
\begin{align}\label{1.13}
b_{ij}^{\alpha\beta}=\frac{\partial}{\partial y_k}(\phi_{kij}^{\alpha\beta}),
\end{align}
where
\begin{align}\label{1.14}
\phi_{kij}^{\alpha\beta}
=\frac{\partial}{\partial y_k}f_{ij}^{\alpha\beta}
-\frac{\partial}{\partial y_i}f_{kj}^{\alpha\beta}
\end{align}
is 1-periodic and belongs to $H^1(Y)$.
It is easy to see that $\phi_{kij}^{\alpha\beta}=-\phi_{ikj}^{\alpha\beta}$. 
This completes the proof.
\end{proof}

The 1-periodic functions $(\phi_{kij}^{\alpha\beta})$ given by Lemma \ref{1.15}
 are called dual correctors for
the family of parabolic operators $\partial_t +\mathcal{L}_\e$, $\e>0$.
As in the elliptic case \cite{Jikov-1994, KLS2}, 
they play an important role in the study of the problem of convergence rates.
Indeed, to establish the main results of this paper, we shall consider the function 
$w_\e = ( w_\e^\alpha)$, where
\begin{equation}\label{w}
\aligned
w_\e ^\alpha (x, t)  = u_\e^\alpha (x, t) -u_0^\alpha  (x, t) &  -\e \chi_j^{\alpha\beta} (x/\e, t/\e^2)
K_\e \left(\frac{\partial u_0^\beta}{\partial x_j}\right)\\
&-\e^2 \phi_{(d+1) ij}^{\alpha\beta} (x/\e, t/\e^2)
\frac{\partial}{\partial x_i} K_\e \left(\frac{\partial u_0^\beta}{\partial x_j}\right),
\endaligned
\end{equation}
and $K_\e : L^2(\Omega_T) \to C_0^\infty(\Omega_T)$ is a linear operator to be chosen later.
The repeated indices $i, j$ in (\ref{w}) are summed from $1$ to $d$.

\begin{thm}\label{Theorem-2.1}
Let $\Omega$ be a bounded Lipschitz domain in $\R^d$ and
$0<T<\infty$.
Let $u_\e \in L^2(0, T; H^1(\Omega))$ and $u_0\in L^2(0, T; H^2(\Omega))$ be solutions of
the initial-Dirichlet problems (\ref{IDP}) and (\ref{IDP-0}), respectively.
Let $w_\e$ be defined by (\ref{w}).
Then for any $\psi \in L^2(0, T; H^1_0(\Omega))$,
\begin{equation}\label{L-w}
\aligned
 \int_0^T  &\big\langle \partial_t w_\e, \psi\big\rangle_{H^{-1}(\Omega)\times H^1_0(\Omega)}
+\iint_{\Omega_T} A^\e \nabla w_\e \cdot \nabla \psi\\
=& \iint_{\Omega_T}
 (\widehat{a}_{ij} -a_{ij}^\e ) 
\left( \frac{\partial u_0}{\partial x_j} - K_\e \left( \frac{\partial u_0}{\partial x_j}\right) \right) \frac{\partial \psi}{\partial x_i}\\
& - \e \iint_{\Omega_T} 
 a_{ij}^\e \cdot \chi_k^\e \cdot \frac{\partial }{\partial x_j}
K_\e \left(\frac{\partial u_0}{\partial x_k} \right) \cdot \frac{\partial \psi}{\partial x_i} \\
& -\e \iint_{\Omega_T}  
 \phi_{kij}^\e \cdot \frac{\partial}{\partial x_i} K_\e \left(\frac{\partial u_0}{\partial x_j}\right)
 \cdot \frac{\partial \psi}{\partial x_k}\\
&-\e^2 \iint_{\Omega_T} 
 \phi_{k(d+1)j}^\e \cdot \partial_t  K_\e \left(\frac{\partial u_0}{\partial x_j}\right)\cdot\frac{\partial \psi}{\partial x_k} \\
&+ \e \iint_{\Omega_T} 
 a_{ij}^\e \cdot \left(\frac{\partial}{\partial x_j} ( \phi_{(d+1) \ell k} ) \right)^\e \cdot \frac{\partial}{\partial x_\ell}
K_\e \left(\frac{\partial u_0}{\partial x_k} \right) \cdot \frac{\partial \psi}{\partial x_i} \\
&+\e^2 \iint_{\Omega_T} 
 a_{ij}^\e \cdot \phi_{(d+1) \ell k } ^\e \cdot \frac{\partial^2 }{\partial x_j \partial x_\ell} 
K_\e \left(\frac{\partial u_0}{\partial x_k}\right) \cdot \frac{\partial \psi}{\partial x_i},
\endaligned
\end{equation}
where we have suppressed superscripts $\alpha, \beta$ for the simplicity of presentation.
The repeated indices  $i, j, k, \ell$ are summed from  $1$ to $d$.
\end{thm}

\begin{proof}
Using (\ref{IDP}) and (\ref{IDP-0}), we see that
$$
\aligned
\big(\partial_t +\mathcal{L}_\e) w_\e
 &=(\mathcal{L}_0 -\mathcal{L}_\e) u_0
 -(\partial_t +\mathcal{L}_\e) \left\{ \e \chi_j^\e K_\e \left(\frac{\partial u_0}{\partial x_j}\right)\right\}\\
 &\qquad \qquad \qquad  -(\partial_t +\mathcal{L}_\e )
\left\{ \e^2 \phi_{(d+1) ij}^\e \frac{\partial}{\partial x_i} K_\e \left(\frac{\partial u_0}{\partial x_j} \right)\right\}\\
&=-\frac{\partial}{\partial x_i}
\left\{ (\widehat{a}_{ij} -a_{ij}^\e ) 
\left( \frac{\partial u_0}{\partial x_j} - K_\e \left( \frac{\partial u_0}{\partial x_j}\right) \right) \right\}\\
&\qquad -\frac{\partial}{\partial x_i}
\left\{ (\widehat{a}_{ij} -a_{ij}^\e) K_\e \left(\frac{\partial u_0}{\partial x_j} \right)\right\}
-(\partial_t +\mathcal{L}_\e) \left\{ \e \chi_j^\e K_\e \left(\frac{\partial u_0}{\partial x_j}\right)\right\}\\
 &\qquad \qquad \qquad  -(\partial_t +\mathcal{L}_\e )
\left\{ \e^2 \phi_{(d+1) ij}^\e \frac{\partial}{\partial x_i} K_\e \left(\frac{\partial u_0}{\partial x_j} \right)\right\}.
\endaligned
$$
By computing the third term in the r.h.s. of the equalities above and using (\ref{b}), we obtain 
$$
\aligned
\big(\partial_t +\mathcal{L}_\e) w_\e
&=-\frac{\partial}{\partial x_i}
\left\{ (\widehat{a}_{ij} -a_{ij}^\e ) 
\left( \frac{\partial u_0}{\partial x_j} - K_\e \left( \frac{\partial u_0}{\partial x_j}\right) \right) \right\}\\
&\qquad  +\frac{\partial}{\partial x_i}
\left\{ b_{ij}^\e K_\e \left(\frac{\partial u_0}{\partial x_j} \right)\right\}
+\e \frac{\partial}{\partial x_i}
\left\{ a_{ij}^\e \cdot \chi_k^\e \cdot \frac{\partial}{\partial x_j} K_\e \left(\frac{\partial u_0}{\partial x_k}\right) \right\}\\
&\qquad - \e \partial_t \left\{ \chi_j^\e K_\e \left(\frac{\partial u_0}{\partial x_j}\right)\right\}
-(\partial_t +\mathcal{L}_\e )
\left\{ \e^2 \phi_{(d+1) ij}^\e \frac{\partial}{\partial x_i} K_\e \left(\frac{\partial u_0}{\partial x_j} \right)\right\}.
\endaligned
$$
In view of (\ref{div}) this gives
\begin{equation}\label{2.2-10}
\aligned
\big(\partial_t +\mathcal{L}_\e) w_\e
&=-\frac{\partial}{\partial x_i}
\left\{ (\widehat{a}_{ij} -a_{ij}^\e ) 
\left( \frac{\partial u_0}{\partial x_j} - K_\e \left( \frac{\partial u_0}{\partial x_j}\right) \right) \right\}\\
& \qquad +\e \frac{\partial}{\partial x_i}
\left\{ a_{ij}^\e \cdot \chi_k^\e \cdot \frac{\partial}{\partial x_j} K_\e \left(\frac{\partial u_0}{\partial x_k}\right) \right\}
+b_{ij}^\e\cdot  \frac{\partial}{\partial x_i} K_\e \left(\frac{\partial u_0}{\partial x_j} \right)\\
&\qquad -\e \chi_j^\e \partial_t K_\e \left(\frac{\partial u_0}{\partial x_j}\right)
-(\partial_t +\mathcal{L}_\e )
\left\{ \e^2 \phi_{(d+1) ij}^\e \frac{\partial}{\partial x_i} K_\e \left(\frac{\partial u_0}{\partial x_j} \right)\right\}.
\endaligned
\end{equation}

Next, by Lemma \ref{1.15}, we may write 
$$
\aligned
& b_{ij}^\e\cdot  \frac{\partial}{\partial x_i} K_\e \left(\frac{\partial u_0}{\partial x_j} \right)
-\e \chi_j^\e \partial_t K_\e \left(\frac{\partial u_0}{\partial x_j}\right)\\
& \qquad =\e \frac{\partial}{\partial x_k}  \Big(\phi_{kij}^\e\Big)\cdot \frac{\partial}{\partial x_i}
K_\e \left(\frac{\partial u_0}{\partial x_j}\right)
+\e^2 \partial_t \Big(\phi_{(d+1) ij}^\e \Big) \cdot \frac{\partial}{\partial x_i} K_\e
\left(\frac{\partial u_0}{\partial x_j} \right)\\
&\qquad \qquad +\e^2 \frac{\partial}{\partial x_k} \Big( \phi_{k (d+1) j }^\e\Big )\cdot
\partial_t K_\e \left(\frac{\partial u_0}{\partial x_j} \right),
\endaligned
$$
where we have also used the fact $\phi_{(d+1)(d+1) j}=0$.
Furthermore, by the skew-symmetry in (\ref{1.10}), we see that
$$
\aligned
& b_{ij}^\e\cdot  \frac{\partial}{\partial x_i} K_\e \left(\frac{\partial u_0}{\partial x_j} \right)
-\e \chi_j^\e \partial_t K_\e \left(\frac{\partial u_0}{\partial x_j}\right)\\
& \qquad =\e \frac{\partial}{\partial x_k}  \left\{ \phi_{kij}^\e \cdot \frac{\partial}{\partial x_i}
K_\e \left(\frac{\partial u_0}{\partial x_j}\right)\right\}
+\e^2 \partial_t \left\{ \phi_{(d+1) ij}^\e  \cdot \frac{\partial}{\partial x_i} K_\e
\left(\frac{\partial u_0}{\partial x_j} \right)\right\} \\
&\qquad \qquad +\e^2 \frac{\partial}{\partial x_k} \left\{  \phi_{k (d+1) j }^\e\cdot
\partial_t K_\e \left(\frac{\partial u_0}{\partial x_j} \right)\right\}.
\endaligned
$$
This, combined with (\ref{2.2-10}), gives the desired equation  (\ref{L-w}).
\end{proof}

The next theorem is concerned with the initial-Neumann problem.

\begin{thm}\label{Theorem-2.2}
Let $\Omega$ be a bounded Lipschitz domain in $\R^d$ and
$0<T<\infty$.
Let $u_\e \in L^2(0, T; H^1(\Omega))$ and $u_0\in L^2(0, T; H^2(\Omega))$ be solutions of
the initial-Neumann  problems (\ref{INP}) and (\ref{INP-0}), respectively.
Let $w_\e$ be defined by (\ref{w}).
Then the equation (\ref{L-w}) holds
for any $\psi \in L^2(0, T; H^1(\Omega))$, if $\langle, \rangle$ in its l.h.s. denotes the pairing 
between $H^1(\Omega)$ and its dual.
\end{thm}

\begin{proof}
It follows from (\ref{INP}) and (\ref{INP-0}) that
$$
\int_0^T \big\langle \partial_t u_\e, \psi \big\rangle +\iint_{\Omega_T} A^\e \nabla u_\e \cdot \nabla \psi
=\int_0^T \big\langle \partial_t u_0, \psi\big \rangle +\iint_{\Omega_T} \widehat{A} \nabla u_0 \cdot \nabla \psi
$$
for any $\psi \in L^2(0, T; H^1(\Omega))$.
This gives 
$$
\aligned
\int_0^T &\big \langle \partial_t w_\e, \psi\big\rangle +\iint_{\Omega_T} A^\e \nabla w_\e \cdot \nabla \psi\\
&=\iint_{\Omega_T} (\widehat{A}-A^\e )\nabla u_0 \cdot \nabla \psi
-\int_0^T \Big\langle (\partial_t +\mathcal{L}_\e) \left\{ \e \chi_j^\e K_\e \left(\frac{\partial u_0}{\partial x_j}\right)\right\}, 
\psi \Big\rangle \\
 & \qquad -\int_0^T \Big\langle 
(\partial_t +\mathcal{L}_\e )
\left\{ \e^2 \phi_{(d+1) ij}^\e \frac{\partial}{\partial x_i} K_\e \left(\frac{\partial u_0}{\partial x_j} \right)\right\},
\psi \Big\rangle,
\endaligned
$$
where we have used the fact $K_\e (\nabla u_0)\in C_0^\infty(\Omega_T)$.
The rest of the proof is similar to that of Theorem \ref{Theorem-2.1}.
We omit the details.
\end{proof}

%%%%%%%%%%%%%%%%%%%%%%%%%%%%%%%%%%

%%%%%%%%%%%%%%%%%%%%%%%%%%%%%%%%

\section{Error estimates  in $L^2(0, T; H^1(\Omega))$}

We begin by introducing a parabolic  smoothing operator.
Fix a nonnegative function  $\theta=\theta (y,s)
\in C_0^\infty(B(0,1))$ such that $\int_{\mathbb{R}^{d+1}}\theta=1$. Define
\begin{equation}\label{1.18}
\aligned
S_\varepsilon(f)(x,t)&=\frac{1}{\varepsilon^{d+2}}\int_{\mathbb{R}^{d+1}}f(x-y,t-s)
\theta(y/\varepsilon,s/\varepsilon^2)\, dyds
\\&=\int_{\mathbb{R}^{d+1}}f(x-\varepsilon y,t-\varepsilon^2 s)\theta(y,s)\, dyds.
\endaligned
\end{equation}

\begin{lemma}\label{lemma-S-1}
Let $S_\varepsilon$ be defined as in (\ref{1.18}). Then
\begin{align}\label{1.19-2}
\| S_\varepsilon (f)\|_{L^2({\mathbb{R}^{d+1}})}\leq \| f \|_{L^2({\mathbb{R}^{d+1}})},
\end{align}
\begin{align}\label{1.19-3}
\e\, \| \nabla S_\varepsilon (f)\|_{L^2({\mathbb{R}^{d+1}})}
+\e^2 \|\nabla^2 S_\e (f)\|_{L^2(\R^{d+1})} \leq C\, \|  f \|_{L^2({\mathbb{R}^{d+1}})},
\end{align}
\begin{align}\label{1.19-4}
\e^2 \| \partial_t S_\varepsilon (f)\|_{L^2({\mathbb{R}^{d+1}})}\leq C\, \| f \|_{L^2({\mathbb{R}^{d+1}})},
\end{align}
where $C$ depends only on $d$.
\end{lemma}

\begin{proof}
This follows easily from  the Plancherel Theorem.
\end{proof}

\begin{lemma}\label{lemma-S-2}
Let $S_\e$ be defined as in (\ref{1.18}). Then
\begin{equation}\label{S-approx}
\|  \nabla S_\e (f) -\nabla f \|_{L^2(\R^{d+1})}
\le C \e \Big\{ \| \nabla^2 f \|_{L^2(\R^{d+1})} 
+\| \partial_t f \|_{L^2(\R^{d+1})} \Big\},
\end{equation}
where $C$ depends only on $d$.
\end{lemma}

\begin{proof}
By the Plancherel Theorem it suffices to show that
$$
|\xi_i \widehat{\theta} (\e \xi^\prime, \e^2 \xi_{d+1}) -\xi_i \widehat{\theta}(0, 0)|
\le C\e \big\{ |\xi^\prime|^2 +|\xi_{d+1}| \big\},
$$
where $1\le i\le d$ and $\xi^\prime=(\xi_1, \dots, \xi_d)\in \R^d$.
Furthermore, by a change of variables, one may assume that $\e=1$.
In this case, if $|\xi^\prime|\ge 1$, then
$$
|\xi_i \widehat{\theta} ( \xi^\prime,  \xi_{d+1}) -\xi_i \widehat{\theta}(0, 0)|
\le C |\xi^\prime|\le C (|\xi^\prime|^2 +|\xi_{d+1}|).
$$
If $|\xi^\prime|\le 1$, we have
$$
|\xi_i \widehat{\theta} ( \xi^\prime,  \xi_{d+1}) -\xi_i \widehat{\theta}(0, 0)|
\le C |\xi^\prime| ( |\xi^\prime| +|\xi_{d+1}|)
\le C ( |\xi^\prime|^2 +|\xi_{d+1}|).
$$
This completes the proof.
\end{proof}

\begin{lemma}\label{lemma-S-3}
Let $g=g(y,s)$ be a 1-periodic function in $(y,s)$. Then
\begin{align}\label{1.22}
\| g^\e S_\varepsilon (f)\|_{L^p({\mathbb{R}^{d+1}})}\leq C\,
\| g \|_{L^p(Y)}  \| f \|_{L^p({\mathbb{R}^{d+1}})}
\end{align}
for any $1\le p<\infty$,
where  $g^\e (x, t)= g(x/\e, t/\e^2)$ and $C$ depends only on $d$ and $p$.
\end{lemma}

\begin{proof}
Note that $S_\e (f) (x, t)= S_1 (f_\e ) (\e^{-1} x, \e^{-2} t)$, where $f_\e (x, t)= f(\e x ,  \e^2 t)$.
As a result, by a change of variables, it suffices to consider the case $\e=1$.
In this case we first use $\int_{\R^{d+1}} \theta =1$ and H\"older's inequality to obtain 
$$
|S_1 (f) (x, t)|^p
\le \int_{\R^{d+1}} |f(y, s)|^p \, \theta (x-y, t-s)\, dyds.
$$
It follows by Fubini's Theorem that
$$
\aligned
\int_{\R^{d+1}} |g(x,t)|^p | S_1 (f)(x, t)|^p\, dx dt
&\le  \sup_{(y, s)\in \R^{d+1}}
\int_{B((y,s), 1)} | g(x, t)|^p \, dx dt \int_{\R^{d+1}} | f(y, s)|^p\, dyds\\
&\le C\,  \|  g \|^p_{L^p(Y)} 
\| f \|_{L^p(\R^{d+1})}^p,
\endaligned
$$
where $C$ depends only on $d$.
This gives (\ref{1.22}) for the case $\e=1$.
\end{proof}

\begin{remark}\label{remark-S}
{\rm
The same argument as in the proof of Lemma \ref{lemma-S-3} also shows that
\begin{equation}\label{S-remark-1}
\aligned
\| g^\e \nabla S_\e (f)\|_{L^p(\R^{d+1})}
 & \le C \e^{-1} \| g\|_{L^p(Y)} \| f\|_{L^p(\R^{d+1})},\\
\| g^\e \partial_t S_\e (f)\|_{L^p(\R^{d+1})}
 & \le C \e^{-2} \| g\|_{L^p(Y)} \| f\|_{L^p(\R^{d+1})}
 \endaligned
 \end{equation}
}
for $1\le p<\infty$, where $C$ depends only on $d$ and $p$.
\end{remark}

Let $\delta \in (2\e, 20\e)$.
Choose  $\eta_1 \in C_0^\infty(\Omega)$ such that $0\le \eta_1\le 1$,
$\eta_1 (x)=1$ if dist$(x, \partial\Omega)\ge 2\delta$, 
$\eta_1 (x)=0$ if dist$(x, \partial\Omega)\le \delta$, and
$|\nabla_x \eta_1|\le C \delta^{-1}$.
Similarly, we choose $\eta_2\in C_0^\infty(0, T)$ such that $0\le \eta_2\le 1$,
$\eta_2 (t) =1$ if $2\delta^2 \le  t\le T-2\delta^2 $,
$\eta_2 (t)=0$ if $t\le \delta^2 $ or $t>T-\delta^2 $, 
and $| \eta_2^\prime (t)|\le C \delta^{-2}$.
We define the operator $K_\e=K_{\e, \delta}: L^2(\Omega_T) \to C_0^\infty (\Omega_T)$ by 
\begin{equation}\label{K}
K_\e (f) (x, t) = S_\e ( \eta_1  \eta_2 f ) (x, t).
\end{equation}

\begin{lemma}\label{main-lemma-3.1}
Let $\Omega$ be a bounded Lipschitz domain in $\R^d$ and $0<T<\infty$. 
Let $u_\varepsilon, u_0\in L^2(0,T;H^1(\Omega))$ be weak solutions of
(\ref{IDP}) and (\ref{IDP-0}), respectively, for some $F\in L^2(\Omega_T)$.
We further assume that $u_0\in L^2(0, T; H^2(\Omega))$ and
$\partial_t u_0 \in L^2(\Omega_T)$.
Let $w_\e$ be defined by (\ref{w}), where the operator $K_\e$ is given by 
(\ref{K}). Then for any $\psi \in L^2(0, T; H_0^1(\Omega))$,
\begin{equation}\label{main-estimate-3.1}
\aligned
& \Big|  \int_0^T  \big\langle (\partial_t +\mathcal{L}_\e) w_\e,
 \psi \big\rangle_{H^{-1}(\Omega) \times H^1_0(\Omega)} \, dt  \Big| \\
&\le 
C
\Big\{ \| u_0\|_{L^2(0, T; H^2(\Omega))}
+\|\partial_t u_0 \|_{L^2(\Omega_T)}
+\e^{-1/2} \| \nabla u_0\|_{L^2(\Omega_{T, 3\delta})} 
 \Big\}\\
&\qquad \qquad
\cdot  \Big\{ \e  \|\nabla \psi\|_{L^2(\Omega_T)}
+\e^{1/2} \|\nabla \psi\|_{L^2(\Omega_{T, 3\delta})} \Big\},
\endaligned
\end{equation}
where 
\begin{equation}\label{o-e}
\Omega_{T, \delta} =\left( \big\{ x\in \Omega: \, \text{\rm dist}(x, \partial\Omega)\le   \delta \big\}
\times (0, T) \right)\cup \left( \Omega \times (0,  \delta^2)\right) \cup \left(\Omega\times (T-\delta^2, T)\right),
\end{equation}
and $C>0$ depends at most on $d$, $m$, $\mu$, $T$ and $\Omega$.
\end{lemma}

\begin{proof}
Using Theorem \ref{Theorem-2.1}, it is not hard to see that the l.h.s. of (\ref{main-estimate-3.1})
is bounded by
\begin{equation}\label{3.4-1}
\aligned
 &C \iint_{\Omega_T}  |\nabla u_0 -K_\e (\nabla u_0)| |\nabla \psi|\\
 &\qquad
 +C \e \iint_{\Omega_T} \Big\{  |\chi^\e| +|\phi^\e| +|(\nabla \phi)^\e|\Big\}  |\nabla K_\e (\nabla u_0)||\nabla \psi|\\
&\qquad
+ C \e^2 \iint_{\Omega_T}
|\phi^\e| \Big\{ |\partial_t K_\e (\nabla u_0)| +|\nabla^2 K_\e (\nabla u_0)|\Big\} |\nabla \psi|\\
&=I_1 +I_2 +I_3,
\endaligned
\end{equation}
where $C$ depends only on $d$, $m$ and $\mu$.
To estimate $I_2$, we note that
\begin{equation}\label{3.4-2}
\nabla K_\e (\nabla u_0)
=\nabla S_\e (\eta_1\eta_2(\nabla u_0))
=S_\e (\nabla (\eta_1 \eta_2)  (\nabla u_0))
+S_\e (\eta_1\eta_2(\nabla^2 u_0)).
\end{equation}
It follows by the Cauchy inequality and Lemma \ref{lemma-S-3} that
$$
\aligned
I_2
 \le  &C \e \left(\iint_{\Omega_T}
| \Big\{  |\chi^\e| +|\phi^\e| +|(\nabla \phi)^\e|\Big\} S_\e (\nabla (\eta_1\eta_2)(\nabla u_0))|^2 \right)^{1/2}
\left(\iint_{\Omega_{T, 3\delta}} |\nabla \psi|^2 \right)^{1/2}\\
&\quad  +C 
\e \left(\iint_{\Omega_T}
| \Big\{  |\chi^\e| +|\phi^\e| +|(\nabla \phi)^\e|\Big\} S_\e ( \eta_1\eta_2(\nabla^2 u_0))|^2 \right)^{1/2}
\left(\iint_{\Omega_T} |\nabla \psi|^2 \right)^{1/2}\\
&\le C \left(\iint_{\Omega_{T, 3\delta}} |\nabla u_0|^2\right)^{1/2}
\left(\iint_{\Omega_{T, 3\delta}} |\nabla \psi|^2\right)^{1/2}\\
&\qquad+C \e  \left(\iint_{\Omega_{T}} |\nabla^2 u_0|^2\right)^{1/2}
\left(\iint_{\Omega_T} |\nabla \psi|^2\right)^{1/2},
\endaligned
$$
where we also have used the observation that $S_\e (\nabla (\eta_1\eta_2)(\nabla u_0))$
is supported in $\Omega_{T, 3\delta}$.
This shows that $I_2$ is bounded by the r.h.s. of (\ref{main-estimate-3.1}).

Next, to handle the term $I_3$, we note that
$$
\aligned
\partial_t K_\e (\nabla u_0)
 &=\partial_t S_\e (\eta_1\eta_2(\nabla u_0))
 =S_\e (\partial_t (\eta_1\eta_2) \nabla  u_0)
+S_\e (\eta_1\eta_2(\nabla \partial_t u_0))\\
& =S_\e (\partial_t (\eta_1\eta_2) \nabla  u_0)
+\nabla S_\e (\eta_1\eta_2 (\partial_t u_0))
-S_\e (\nabla (\eta_1\eta_2) (\partial_t u_0)),
\endaligned
$$
and
$$
\nabla^2 K_\e (\nabla u_0)
=\nabla S_\e (\nabla (\eta_1\eta_2) (\nabla u_0))
+\nabla S_\e (\eta_1 \eta_2( \nabla^2 u_0)).
$$
As in the case of $I_2$, by the Cauchy inequality and Remark \ref{remark-S} , this gives
$$
\aligned
I_3 \le  &C \left(\iint_{\Omega_{T, 3\delta}} |\nabla u_0|^2\right)^{1/2}
\left(\iint_{\Omega_{T, 3\delta}} |\nabla \psi|^2 \right)^{1/2}\\
&+ C\e \left(\iint_{\Omega_{T}} |\partial_t  u_0|^2\right)^{1/2}
\left(\iint_{\Omega_{T}} |\nabla \psi|^2 \right)^{1/2}\\
&+C\e \left(\iint_{\Omega_{T}} |\nabla^2 u_0|^2\right)^{1/2}
\left(\iint_{\Omega_{T}} |\nabla \psi|^2 \right)^{1/2},
\endaligned
$$
which is bounded by the r.h.s. of (\ref{main-estimate-3.1}).

Finally, to estimate $I_1$, we observe that 
\begin{equation}\label{3.4-3}
\aligned
I_1
\le & C\iint_{\Omega_{T, 2\delta}}\Big\{  |\nabla u_0| +S_\e (\eta_1\eta_2 |\nabla u_0|)\Big\} |\nabla \psi|
+C\iint_{\Omega_T\setminus \Omega_{T, 2\delta}}
| (\nabla u_0 -S_\e (\nabla u_0)) | |\nabla \psi|\\
&\le C \left(\iint_{\Omega_{T, 3\delta}} |\nabla u_0|^2\right)^{1/2} 
\left(\iint_{\Omega_{T, 3\delta}} |\nabla \psi|^2 \right)^{1/2}\\
&\qquad 
+ C \left(\iint_{\Omega\setminus \Omega_{T, 2\delta}} |\nabla u_0 -S_\e (\nabla u_0)|^2\right)^{1/2}
\left(\iint_{\Omega_T} |\nabla \psi|^2 \right)^{1/2}.
\endaligned
\end{equation}
To treat the second term in the r.h.s. of (\ref{3.4-3}),
we extend $u_0$ to a function $\widetilde{u}_0$ in $\R^{d+1}$ such that 
$$
\left(\iint_{\R^{d+1}} |\nabla^2 \widetilde{u}_0|^2\right)^{1/2}
+\left(\iint_{\R^{d+1}} |\partial_t \widetilde{u}_0|^2\right)^{1/2}
\le C \Big\{ \| u_0\|_{L^2(0, T; H^2(\Omega))}
+\|\partial_t u_0\|_{L^2(\Omega_T)} \Big\},
$$
using the Calder\'on's  extension theorem.
It follows that
$$
\aligned
\left(\iint_{\Omega\setminus \Omega_{T, 2\delta}} |\nabla u_0 -S_\e (\nabla u_0)|^2\right)^{1/2}
& \le \left(\iint_{\R^{d+1}}
 |\nabla \widetilde{u}_0 -S_\e (\nabla \widetilde{u}_0)|^2\right)^{1/2}\\
 &\le C\e \Big\{ \|\nabla^2 \widetilde{u}_0\|_{L^2(\R^{d+1})}
 +\|\partial_t \widetilde{u}_0\|_{L^2(\R^{d+1})}\Big\} \\
 &\le C\e \Big\{ \| u_0\|_{L^2(0, T; H^2(\Omega))}
+\|\partial_t u_0\|_{L^2(\Omega_T)} \Big\},
\endaligned
$$
where we have used Lemma \ref{lemma-S-2} for the second inequality.
As a result, we see that
$I_1$ is also bounded by the r.h.s. of (\ref{main-estimate-3.1}).
This completes the proof.
\end{proof}

\begin{remark}\label{remark-u-0}
{\rm 
Let $\Omega^\delta=\big\{ x\in \Omega: \text{dist}(x, \partial\Omega)< \delta \big\}$.
Then
\begin{equation}\label{boundary-estimate}
\int_{\Omega^{\delta}} |\nabla u_0|^2
\le C  \delta  \| \nabla u_0\|^2_{H^1(\Omega)}
\end{equation}
(see e.g. \cite{SZ-2015} for a proof).
It follows that
$$
\aligned
\|\nabla u_0\|_{L^2(\Omega_{T, 3\delta})}
 &\le \left(\int_0^T \int_{\Omega^{3\delta}} |\nabla u_0|^2 \right)^{1/2}
+\left(\int_0^{c\e^2} \int_\Omega |\nabla u_0|^2 \right)^{1/2}
+\left(\int_{T-c\e^2}^T \int_\Omega |\nabla u_0|^2\right)^{1/2}\\
&\le C \e^{1/2}
\left\{ \| u_0\|_{L^2(0, T; H^2(\Omega))} +
\sup_{\e^2<t<T}
\left(\frac{1}{\e} \int_{t-\e^2}^t \int_\Omega |\nabla u_0|^2 \right)^{1/2} \right\}.
\endaligned
$$
}
\end{remark}

The next theorem provides an $O(\sqrt{\e})$ error estimate
in $L^2(0, T; H^1_0(\Omega))$ for the initial-Dirichlet problem (\ref{IDP}).

\begin{thm}\label{IDP-H-1}
Let $w_\e$ be defined by (\ref{w}).
Under the same assumptions as in Lemma \ref{main-lemma-3.1}, we have
\begin{equation}\label{estimate-IDP-H-1}
\aligned
&\|\nabla w_\e\|_{L^2(\Omega_T)}\\
&\le C \sqrt{\e}
\left\{ \| u_0\|_{L^2(0, T; H^2(\Omega))}
+\|\partial_t u_0\|_{L^2(\Omega_T)}
+\sup_{\e^2<t<T}
\left(\frac{1}{\e} \int_{t-\e^2}^t \int_\Omega |\nabla u_0|^2 \right)^{1/2} \right\},
\endaligned
\end{equation}
where $C$ depends at most on $d$, $m$, $\mu$, $T$ and $\Omega$.
\end{thm}

\begin{proof}
Note that $w_\e \in L^2(0, T; H^1_0(\Omega))$ and $w_\e =0$ on $\Omega \times \{ t=0\}$.
It follows that
$$
\aligned
  &\mu \iint_{\Omega_T} |\nabla w_\e|^2
  \le \int_0^T \big \langle (\partial_t +\mathcal{L}_\e) w_\e, w_\e\big\rangle_{H^{-1}(\Omega) \times H^1_0(\Omega) }\\
  &\le C \sqrt{\e} \|\nabla w_\e\|_{L^2(\Omega_T)}
  \Big\{ \| u_0\|_{L^2(0, T; H^2(\Omega))}
+\|\partial_t u_0 \|_{L^2(\Omega_T)}
+\e^{-1/2} \| \nabla u_0\|_{L^2(\Omega_{T, 60\e})} 
 \Big\},
 \endaligned
 $$
 where we have used Lemma \ref{main-lemma-3.1} for the last step.
 This, together with Remark \ref{remark-u-0}, gives (\ref{estimate-IDP-H-1}).
\end{proof}

Next we consider  the initial-Neumann problem (\ref{INP}).

\begin{lemma}\label{main-lemma-3.2}
Let $\Omega$ be a bounded Lipschitz domain in $\R^d$ and $0<T<\infty$. 
Let $u_\varepsilon, u_0\in L^2(0,T;H^1(\Omega))$ be weak solutions of
the initial-Neumann problems (\ref{INP}) and (\ref{INP-0}), respectively,
for some $F\in L^2(\Omega_T)$.
We further assume that $u_0\in L^2(0, T; H^2(\Omega))$ and that
$\partial_t u_0 \in L^2(\Omega_T)$.
Let $w_\e$ be defined by (\ref{w}), where the operator $K_\e$ is given by 
(\ref{K}). Then for any $\psi \in L^2(0, T; H^1(\Omega))$,
\begin{equation}\label{main-estimate-3.2}
\aligned
&  \Big| \int_0^T  \big\langle \partial_t  w_\e, \psi\big \rangle
+\iint_{\Omega_T} A^\e \nabla w_\e \cdot \nabla \psi   \Big| \\
&\le 
C
\Big\{ \| u_0\|_{L^2(0, T; H^2(\Omega))}
+\|\partial_t u_0 \|_{L^2(\Omega_T)}
+\e^{-1/2} \| \nabla u_0\|_{L^2(\Omega_{T, 3\delta})} 
 \Big\}\\
&\qquad \qquad
\cdot  \Big\{ \e  \|\nabla \psi\|_{L^2(\Omega_T)}
+\e^{1/2} \|\nabla \psi\|_{L^2(\Omega_{T, 3\delta})} \Big\},
\endaligned
\end{equation}
where $\langle, \rangle$ denotes the pairing between $H^1(\Omega)$ and its dual.
The constant  $C>0$ depends at most on $d$, $m$, $\mu$, $T$ and $\Omega$.
\end{lemma}

\begin{proof}
This follows from Theorem \ref{Theorem-2.2} by the same argument as in the proof of Lemma 
\ref{main-lemma-3.1}.
\end{proof}

\begin{thm}\label{INP-H-1}
Let $w_\e$ be defined by (\ref{w}).
Under the same assumptions as in Lemma \ref{main-lemma-3.2}, we have
\begin{equation}\label{estimate-INP-H-1}
\aligned
&\|\nabla w_\e\|_{L^2(\Omega_T)}\\
&\le C \sqrt{\e}
\left\{ \| u_0\|_{L^2(0, T; H^2(\Omega))}
+\|\partial_t u_0\|_{L^2(\Omega_T)}
+\sup_{\e^2<t<T}
\left(\frac{1}{\e} \int_{t-\e^2}^t \int_\Omega |\nabla u_0|^2 \right)^{1/2} \right\},
\endaligned
\end{equation}
where $C$ depends at most on $d$, $m$, $\mu$, $T$ and $\Omega$.
\end{thm}

\begin{proof}
As in the proof of Theorem \ref{IDP-H-1},
this follows from Lemma \ref{main-lemma-3.2}
by letting $\psi =w_\e$.
\end{proof}

\begin{remark}\label{remark-3.10}
{\rm In the case of $u_\e =u_0=0$ or
$\frac{\partial u_\e}{\partial \nu_\e}=\frac{\partial u_0}{\partial \nu_0} =0$ on $\partial\Omega\times (0,T)$,
we may bound the third term in the r.h.s. of (\ref{estimate-INP-H-1}) as follows.
Note that
\begin{equation}\label{3.10-1}
\int_\Omega \widehat{A}\nabla u_0 \cdot \nabla u_0
=-\int_\Omega \partial_t u_0 \cdot u_0 +\int_\Omega F\cdot u_0.
\end{equation}
It follows that
$$
\aligned
\mu \int_{t-\e^2}^t \int_\Omega
|\nabla u_0|^2 
&\le  \int_{t-\e^2}^t \int_\Omega |\partial_t u_0||u_0| +
 \int_{t-\e^2}^t \int_\Omega  |F||u_0| \\
& \le \Big\{  \| \partial_t u_0\|_{L^2(\Omega_T)} +\| F\|_{L^2(\Omega_T)} \Big\}
\left(\int_{t-\e^2}^t \int_\Omega |u_0|^2\right)^{1/2}\\
&\le  \e \Big\{  \| \partial_t u_0\|_{L^2(\Omega_T)} +\| F\|_{L^2(\Omega_T)} \Big\}
\sup_{0<t<T} \| u_0(\cdot, t)\|_{L^2(\Omega)}.
\endaligned
$$
This, together with the standard energy estimates, gives
\begin{equation}\label{3.10-2}
\sup_{\e^2<t<T}
\left(\frac{1}{\e} \int_{t-\e^2}^t \int_\Omega |\nabla u_0|^2 \right)^{1/2} 
\le C \Big\{
\|\partial_t u_0\|_{L^2(\Omega_T)}
+\| F\|_{L^2(\Omega_T)} + \| h \|_{L^2(\Omega)} \Big\},
\end{equation}
where $C$ depends only on $d$, $m$, $\mu$ and $\Omega$.
As a result, for both the initial-Dirichlet problem (\ref{IDP})
and the initial-Neumann problem (\ref{INP}), if $g=0$ on $\partial\Omega \times (0, T)$,
 then
\begin{equation}\label{3.10-3}
\| \nabla w_\e\|_{L^2(\Omega_T)}
\le C \sqrt{\e}
\Big\{ \| u_0\|_{L^2(0, T; H^2(\Omega))}
+\| F\|_{L^2(\Omega_T)}
+\| h\|_{L^2(\Omega)} \Big\},
\end{equation}
where we have used the fact 
$$
\|\partial _t u_0\|_{L^2(\Omega_T)}
\le C \left\{ \|\nabla^2 u_0\|_{L^2(\Omega_T)} +\| F\|_{L^2(\Omega_T)} \right\}.
$$
In particular, if $\Omega$ is $C^{1,1}$, $g=0$ on $\partial\Omega\times (0, T)$ and
 $h=0$ on $\Omega$, then
 \begin{equation}\label{3.10-4}
 \| \nabla w_\e\|_{L^2(\Omega_T)}
 \le C \sqrt{\e} \| F \|_{L^2(\Omega_T)}.
 \end{equation}
 To see this, we use the well-known estimate
 \begin{equation}\label{3.10-5}
 \| u_0\|_{L^2(0, T; H^2(\Omega))} \le C\, \| F\|_{L^2(\Omega_T)},
 \end{equation}
 which may be proved by using the partial Fourier transform in the $t$ variable and
 reducing the problem to the $H^2$ estimate for the elliptic operator  $\mathcal{L}_0$ in $C^{1,1}$ domains.
 We also note that in the case that $g=0$ on $\partial\Omega\times (0, T)$ and $h\in H^1(\Omega; \R^m)$,
 if $\mathcal{L}^*_0=\mathcal{L}_0$ and $\Omega$ is $C^{1,1}$, then
 \begin{equation}\label{3.10-10}
 \| u_0\|_{L^2(0, T; H^2(\Omega))}
 \le C \Big\{ \| F\|_{L^2(\Omega_T)} +\| h\|_{H^1(\Omega)} \Big\}.
 \end{equation}
 This may be proved by using integration by parts as well as  $H^2$ estimates for
 $\mathcal{L}_0$ \cite{Lady}.
 % holds if $\Omega$ is $C^{1,1}$ \cite{Lieberman}.
}
\end{remark}

%%%%%%%%%%%%%%%%%%%%%%%%%%%%%%%%%%%%%%%%%%%%%%

%%%%%%%%%%%%%%%%%%%%%%%%%%%%%%%%%%%%%%%%%%%%%%%

\section{Proof of Theorems \ref{main-theorem-1} and \ref{main-theorem-2}}

In this section we study the convergence rates in $L^2(\Omega_T)$ and
give the proof of Theorems \ref{main-theorem-1} and \ref{main-theorem-2}.
Throughout the section we will assume that $\Omega$ is a bounded
$C^{1,1}$ domain in $\R^d$.

We first consider the initial-Dirichlet problem.
Let $A^*$ denote the adjoint of $A$; i.e., $A^*=( a^{*\alpha\beta}_{ij})$
with $a_{ij}^{*\alpha\beta} (y, s)=a_{ji}^{\beta\alpha}(y,s)$.
For $G\in L^2(\Omega_T)$, let $v_\e$ be the weak solution to 
\begin{equation}\label{IDP-dual}
\left\{
\aligned
(-\partial_t +\mathcal{L}^*_\e) v_\e  & = G &\quad & \text{ in } \Omega \times (0, T),\\
v_\e &=0 &\quad &\text{ on } \partial\Omega \times (0, T),\\
v_\e & =0& \quad & \text{ on } \Omega \times \{ t=T\},
\endaligned
\right.
\end{equation}
where $\mathcal{L}^*_\e =-\text{\rm div} (A^{*\e} (x, t)\nabla )$ denotes the adjoint of $\mathcal{L}_\e$,
and $v_0$ the weak solution to 
\begin{equation}\label{IDP-0-dual}
\left\{
\aligned
(-\partial_t +\mathcal{L}^*_0) v_0  & = G &\quad & \text{ in } \Omega \times (0, T),\\
v_0 &=0 &\quad &\text{ on } \partial\Omega \times (0, T),\\
v_0 & =0& \quad & \text{ on } \Omega \times \{ t=T\},
\endaligned
\right.
\end{equation}
where $\mathcal{L}_0^*=-\text{\rm div} (\widehat{A}^* \nabla )$.
Observe that $v_\e (x, T-t)$ and $v_0 (x, T-t)$ are solutions of the
initial-Dirichlet problems of (\ref{IDP}) and (\ref{IDP-0}), respectively, with 
coefficient matrix $A(x/\e,t/\e^2)$ replaced by $A^*(x/\e, (T-t)/\e^2)$, and with $g=0$ and $h=0$. 
Also note that $A^* (y, T-s)$ satisfies the same ellipticity and periodicity conditions as $A(y,s)$.

\begin{lemma}\label{lemma-4.0}
Let $v_0$ be the weak solution to (\ref{IDP-0-dual}).
Then
\begin{equation}\label{4.0-0}
\| \nabla v_0\|_{L^2(\Omega_T)} +\delta^{-1/2} \|\nabla v_0\|_{L^2(\Omega_{T, \delta})}
\le C \| G \|_{L^2(\Omega_T)},
\end{equation}
where $\delta \in (2\e, 20\e)$ and  $C$ depends at most on $d$, $m$, $\mu$, $T$ and $\Omega$.
\end{lemma}

\begin{proof}
The estimate  for $\|\nabla v_0\|_{L^2(\Omega_T)}$ follows directly from the energy estimate,
while the estimate for $\delta^{-1/2} \|\nabla v_0\|_{L^2(\Omega_{T, \delta})}$
is proved in Remarks \ref{remark-u-0} and  \ref{remark-3.10}.
\end{proof}

Let
\begin{equation}\label{z}
\aligned
z_\e (x, t)= v_\e (x, T-t) -v_0 (x, T-t)-  & \e \chi_{T, j}^{*\e} S_\e \left(\widetilde{\eta}(x,t) \frac{\partial v_0}{\partial x_j} 
 (x, T-t)\right) \\
- &\e^2 \phi_{T, (d+1)ij}^{*\e} \frac{\partial}{\partial x_i}
S_\e \left( \widetilde{\eta}(x,t) \frac{\partial v_0}{\partial x_j} (x, T-t)\right),
\endaligned
\end{equation}
where $\chi_T^*$ and $\phi_T^*$ denote the correctors and dual correctors, respectively,
 for the family of parabolic operators $\partial_t +\text{div} (A^*(x/\e, (T-t)/\e^2)\nabla )$, $\e>0$.
The cut-off  function $\widetilde{\eta}$ in (\ref{z}) is chosen so that
$\widetilde{\eta} (x,t)=0$ if $(x, t)\in \Omega_{T, 10\e}$, $\eta(x, t)=1$
if $(x, t)\in \Omega_T \setminus \Omega_{T, 15\e}$,
$|\nabla \widetilde{\eta} |\le C \e^{-1}$ and $|\partial_t \widetilde{\eta}|\le C \e^{-2}$.

\begin{lemma}\label{lemma-4.1}
Let $z_\e$ be defined by (\ref{z}). Then
\begin{equation}\label{4.1-0}
\| \nabla z_\e\|_{L^2(\Omega_T)} \le C \sqrt{\e} \| G\|_{L^2(\Omega_T)},
\end{equation}
where $C$ depends at most on $d$, $m$, $\mu$, $T$ and $\Omega$.
\end{lemma}

\begin{proof}
Since $A^* (y, T-s)$ satisfies the same ellipticity and periodicity conditions as $A(y,s)$ and
$\Omega$ is $C^{1,1}$,
this follows from the estimate (\ref{3.10-4}).
\end{proof}

We are in a position to give the proof of Theorem \ref{main-theorem-1}.

\begin{proof}[\bf Proof of Theorem \ref{main-theorem-1}]

Let $u_\e\in L^2(0, T; H^1(\Omega))$ and $u_0\in L^2(0, T; H^2(\Omega))$ be solutions of
(\ref{IDP}) and (\ref{IDP-0}), respectively.
Let $G\in L^2(\Omega_T)$. By duality it suffices to show that
\begin{equation}\label{4.3-1}
\aligned
 & \Big| \iint_{\Omega_T} (u_\e -u_0) \cdot G \Big|\\
&\le 
C \e \| G\|_{L^2(\Omega_T)}
\left\{ \| u_0\|_{L^2(0, T; H^2(\Omega))}
+\|\partial_t u_0\|_{L^2(\Omega_T)}
+\sup_{\e^2<t<T}
\left(\frac{1}{\e} \int_{t-\e^2}^t \int_\Omega |\nabla u_0|^2 \right)^{1/2} \right\}.
\endaligned
\end{equation}
Let $w_\e$ be defined by (\ref{w}), with $\delta=2\e$.
Since
$$
\| \chi^\e  K_\e (\nabla u_0)\|_{L^2(\Omega_T)}
+\e \| \phi^\e  \nabla K_\e (\nabla u_0)\|_{L^2(\Omega_T)}
\le C \|\nabla u_0\|_{L^2(\Omega_T)},
$$
we only need to prove that $\displaystyle | \iint_{\Omega_T} w_\e \cdot G|$ is bounded
by the r.h.s. of (\ref{4.3-1}). 

To this end we write
\begin{equation}\label{4.3-2}
\aligned
  &\iint_{\Omega_T} w_\e \cdot G
 =  \int_0^T \big\langle \partial_t w_\e, v_\e \big\rangle +\iint_{\Omega_T} A^\e \nabla w_\e \cdot \nabla v_\e\\
 &= \left\{  \int_0^T \big\langle \partial_t w_\e, z_\e (\cdot, T-t)\big\rangle
 +\iint_{\Omega_T} A^\e \nabla w_\e \cdot \nabla z_\e (x, T-t)\right\}\\
 &\quad +\left\{  \int_0^T\big \langle \partial_t w_\e, v_0 \big\rangle
 +\iint_{\Omega_T} A^\e \nabla w _\e\cdot \nabla v_0\right\} \\
 &\quad +\left\{ \int_0^T\big \langle \partial_t w_\e, v_\e -v_0 -z_\e (\cdot, T-t)\big \rangle
 + \iint_{\Omega_T} A^\e \nabla w_\e \cdot
 \big\{ v_\e -v_0 - z_\e (\cdot, T-t) \big\}\right\}\\
 &=J_1 +J_2 +J_3,
 \endaligned
 \end{equation}
 where $\langle, \rangle$ denotes the pairing between $H^1_0(\Omega)$ and its dual $H^{-1}(\Omega)$.
 We shall use Lemma \ref{main-lemma-3.1} to bound $J_1$, $J_2$ and $J_3$.
 
 For the term $J_1$, it follows by Lemma \ref{main-lemma-3.1} that
 \begin{equation}\label{4.3-3}
 \aligned
 |J_1| & \le C \sqrt{\e} 
 \Big \{ \| u_0\|_{L^2(0, T; H^2(\Omega))} +\|\partial_t u_0\|_{L^2(\Omega_T)}
 +\e^{-1/2} \|\nabla u_0\|_{L^2(\Omega_{T, 6\e})} \Big\}
 \| \nabla z_\e\|_{L^2(\Omega_T)}\\
 & \le C \e \Big \{ \| u_0\|_{L^2(0, T; H^2(\Omega))} +\|\partial_t u_0\|_{L^2(\Omega_T)}
 +\e^{-1/2} \|\nabla u_0\|_{L^2(\Omega_{T, 6\e})} \Big\} \| G\|_{L^2(\Omega_T)},
 \endaligned
 \end{equation}
 where we have used Lemma \ref{lemma-4.1} for the last step.
 
 Next, for $J_2$, we obtain 
 \begin{equation}\label{4.3-4}
 \aligned
 |J_2|
 &\le C 
 \Big \{ \| u_0\|_{L^2(0, T; H^2(\Omega))} +\|\partial_t u_0\|_{L^2(\Omega_T)}
 +\e^{-1/2} \|\nabla u_0\|_{L^2(\Omega_{T, 6\e})} \Big\}\\
&\qquad\qquad \qquad\qquad
\cdot
 \Big\{ \e \| \nabla v_0\|_{L^2(\Omega)} +\e^{1/2} \|\nabla v_0\|_{L^2(\Omega_{T, 6\e})} \Big\}\\
 &\le C \e \Big \{ \| u_0\|_{L^2(0, T; H^2(\Omega))} +\|\partial_t u_0\|_{L^2(\Omega_T)}
 +\e^{-1/2} \|\nabla u_0\|_{L^2(\Omega_{T, 6\e})} \Big\}\| G\|_{L^2(\Omega_T)},
 \endaligned
 \end{equation}
 where we have used Lemma \ref{lemma-4.0} for the last inequality.
 
 To estimate $J_3$, we note that 
 $v_\e -v_0 -z_\e(x, T-t)$ is supported in $\Omega_T\setminus \Omega_{T, 10\e}$ and
 in view of (\ref{z}) and Lemmas \ref{lemma-S-1} and \ref{lemma-S-3},
 $$
 \|\nabla (v_\e -v_0 -z_\e (x, T-t))\|_{L^2(\Omega_T)}
 \le C \| \nabla v_0\|_{L^2(\Omega_T)}\le C \| G\|_{L^2(\Omega_T)}.
 $$
 It follows by Lemma \ref{main-lemma-3.1} that
 $$
 |J_3|
 \le C \e \Big \{ \| u_0\|_{L^2(0, T; H^2(\Omega))} +\|\partial_t u_0\|_{L^2(\Omega_T)}
 +\e^{-1/2} \|\nabla u_0\|_{L^2(\Omega_{T, 6\e})} \Big\} \| G\|_{L^2(\Omega_T)}.
 $$
 This, together with (\ref{4.3-3}) and (\ref{4.3-4}), shows that 
 $$
 \aligned
 &\Big| \iint_{\Omega_T} w_\e \cdot G\Big|\\
& \le C \e\Big \{ \| u_0\|_{L^2(0, T; H^2(\Omega))} +\|\partial_t u_0\|_{L^2(\Omega_T)}
 +\e^{-1/2} \|\nabla u_0\|_{L^2(\Omega_{T, 6\e})} \Big\} \| G\|_{L^2(\Omega_T)}\\
 & \le C \e\left \{ \| u_0\|_{L^2(0, T; H^2(\Omega))} +\|\partial_t u_0\|_{L^2(\Omega_T)}
 +\sup_{\e^2<t<T}
\left(\frac{1}{\e} \int_{t-\e^2}^t \int_\Omega |\nabla u_0|^2 \right)^{1/2}
 \right\} \| G\|_{L^2(\Omega_T)},
 \endaligned
$$
which completes the proof.
\end{proof}

Finally, we give the proof of Theorem \ref{main-theorem-2}

\begin{proof}[\bf Proof of Theorem \ref{main-theorem-2}]
The proof of Theorem \ref{main-theorem-2}
is similar to that of Theorem \ref{main-theorem-1}.
Indeed, let $u_\e\in L^2(0, T; H^1(\Omega))$ and $u_0\in L^2(0,T; H^2(\Omega))$
be solutions of (\ref{INP}) and (\ref{INP-0}), respectively.
Let $w_\e$ be defined as in (\ref{w}), with $\delta=2\e$. 
To estimate $\| u_\e -u_0\|_{L^2(\Omega_T)}$, we consider 
$\displaystyle \iint_{\Omega_T} w_\e \cdot G$, where $G\in L^2(\Omega_T)$.
Let $v_\e$ be the weak solution to 
\begin{equation}\label{INP-dual}
\left\{
\aligned
(-\partial_t +\mathcal{L}^*_\e) v_\e  & = G &\quad & \text{ in } \Omega \times (0, T),\\
\frac{\partial v_\e}{\partial \nu^*_\e} &=0 &\quad &\text{ on } \partial\Omega \times (0, T),\\
v_\e & =0& \quad & \text{ on } \Omega \times \{ t=T\},
\endaligned
\right.
\end{equation}
and $v_0$ the weak solution to 
\begin{equation}\label{INP-dual-0}
\left\{
\aligned
(-\partial_t +\mathcal{L}^*_0) v_0  & = G &\quad & \text{ in } \Omega \times (0, T),\\
\frac{\partial v_0}{\partial \nu^*_0} &=0 &\quad &\text{ on } \partial\Omega \times (0, T),\\
v_0 & =0& \quad & \text{ on } \Omega \times \{ t=T\},
\endaligned
\right.
\end{equation}
where $\frac{\partial v_\e}{\partial\nu^*_\e}$ and
$\frac{\partial v_0}{\partial\nu_0^*}$ denote the conormal derivatives
associated with the operators $\mathcal{L}_\e^*$ and $\mathcal{L}_0^*$,
respectively. Let 
$z_\e$ be defined as before. 
Note that estimates in Lemmas \ref{lemma-4.0} and \ref{lemma-4.1} continue to hold.
Moreover, by (\ref{INP-dual}), we have
$$
\iint_{\Omega_T} w_\e \cdot G
=\int_0^T \big\langle \partial_t w_\e, v_\e\big\rangle 
+\iint_{\Omega_T} A^\e \nabla w_\e \cdot \nabla v_\e,
$$
where $\langle, \rangle$ denotes the pairing between $H^1(\Omega)$ and its dual.
With Lemma \ref{main-lemma-3.2}  at our disposal, the rest of the proof
is exactly the same as that of Theorem \ref{main-theorem-1}.
We omit the details.
\end{proof}

 \bibliographystyle{amsplain}
 
\bibliography{convergence.bbl}

\providecommand{\bysame}{\leavevmode\hbox to3em{\hrulefill}\thinspace}
\providecommand{\MR}{\relax\ifhmode\unskip\space\fi MR }
% \MRhref is called by the amsart/book/proc definition of \MR.
\providecommand{\MRhref}[2]{%
  \href{http://www.ams.org/mathscinet-getitem?mr=#1}{#2}
}
\providecommand{\href}[2]{#2}
\begin{thebibliography}{10}

\bibitem{Armstrong-Shen-2016}
S.N. Armstrong and Z.~Shen, \emph{Lipschitz estimates in almost-periodic
  homogenization}, Comm. Pure Appl. Math. (to appear).

\bibitem{Armstrong-Smart-2014}
S.N. Armstrong and C.K. Smart, \emph{Quantitative stochastic homogenization of
  convex integral functionals}, Ann. Sci. \'Ec. Norm. Sup\'er (to appear).

\bibitem{AL-1987}
M.~Avellaneda and F.~Lin, \emph{Compactness methods in the theory of
  homogenization}, Comm. Pure Appl. Math. \textbf{40} (1987), 803--847.

\bibitem{bensoussan-1978}
A.~Bensoussan, J.-L. Lions, and G.C. Papanicolaou, \emph{Asymptotic {A}nalysis
  for {P}eriodic {S}tructures}, North Holland, 1978.

\bibitem{Geng-Shen-2015}
J.~Geng and Z.~Shen, \emph{Uniform regularity estimates in parabolic
  homogenization}, Indiana Univ. Math. J. \textbf{64} (2015), 697--733.

\bibitem{Griso-2004}
G.~Griso, \emph{Error estimate and unfolding for periodic homogenization},
  Asymptot. Anal. \textbf{40} (2004), 269--286.

\bibitem{Griso-2006}
\bysame, \emph{Interior error estimate for periodic homogenization}, Anal.
  Appl. (Singap.) \textbf{4} (2006), no.~1, 61--79.

\bibitem{Gu-2015}
S.~Gu, \emph{{Convergence rates in homogenization of Stokes systems}},
  arXiv:1508.04203 (2015).

\bibitem{Jikov-1994}
V.V. Jikov, S.M. Kozlov, and O.A. Oleinik, \emph{Homogenization of
  {D}ifferential {O}perators and {I}ntegral {F}unctionals}, Springer-Verlag,
  Berlin, 1994.

\bibitem{KLS2}
C.~Kenig, F.~Lin, and Z.~Shen, \emph{Convergence rates in ${L}^2$ for elliptic
  homogenization problems}, Arch. Rational Mech. Anal. \textbf{203} (2012),
  no.~3, 1009--1036.

\bibitem{KLS4}
\bysame, \emph{{Estimates of eigenvalues and eigenfunctions in periodic
  homogenization}}, J. Eur. Math. Soc. \textbf{15} (2013), 1901--1925.

\bibitem{KLS3}
\bysame, \emph{Periodic homogenization of {G}reen and {N}eumann functions},
  Comm. Pure Appl. Math. \textbf{67} (2014), 1219--1262.

\bibitem{Lady}
Q.A. Ladyzenskaja, Solonnikov V.A., and U.N. Uralceva, \emph{{Linear and
  Quasi-linear Equations of Parabolic Type}}, Amer. Math. Soc., 1968.

\bibitem{Onofrei-2007}
D.~Onofrei and B.~Vernescu, \emph{Error estimates for periodic homogenization
  with non-smooth coefficients}, Asymptot. Anal. \textbf{54} (2007), 103--123.

\bibitem{Shen-Boundary-2015}
Z.~Shen, \emph{Boundary estimates in elliptic homogenization}, arXiv:1505.02525
  (2015).

\bibitem{SZ-2015}
Z.~Shen and J.~Zhuge, \emph{Convergence rates in periodic homogenization of
  systems of elasticity with mixed boundary conditions}, arXiv:1512.00823
  (2015).

\bibitem{Suslina-2004}
T.A. Suslina, \emph{On the averaging of periodic parabolic systems}, Funct.
  Anal. Appl. \textbf{38} (2004), 309--312.

\bibitem{Suslina-2012}
\bysame, \emph{{Homogenization of the elliptic Dirichlet problem: operator
  error estimates in $L_2$}}, Mathematika \textbf{59} (2013), no.~2, 463--476.

\bibitem{Suslina-2013}
\bysame, \emph{Homogenization of the {N}eumann problem for elliptic systems
  with periodic coefficients}, SIAM J. Math. Anal. \textbf{45} (2013), no.~6,
  3453--3493.

\bibitem{Suslina-2015}
\bysame, \emph{Homogenization of solutions of initial boundary value problems
  for parabolic systems}, Func. Anal. Appl. \textbf{49} (2015), 72--76.

\bibitem{Zhikov-2006}
V.V. Zhikov and S.E. Pastukhova, \emph{Estimates of homogenization for a
  parabolic equation with periodic coefficients}, Russ. J. Math. Phys.
  \textbf{13} (2006), 224--237.

\end{thebibliography}

\bigskip

\begin{flushleft}
Jun Geng,
School of Mathematics and Statistics,
Lanzhou University,
Lanzhou, P.R. China.

E-mail:gengjun@lzu.edu.cn
\end{flushleft}

\begin{flushleft}
Zhongwei Shen,
Department of Mathematics,
University of Kentucky,
Lexington, Kentucky 40506,
USA.

E-mail: zshen2@uky.edu
\end{flushleft}

\bigskip

\medskip

\end{document}